\documentclass[11pt]{amsart}

\usepackage{amssymb}
\usepackage{amsfonts}
\usepackage{amsmath}
\usepackage{graphicx}
\usepackage[usenames,dvipsnames]{xcolor}
\usepackage{amsthm}
\usepackage{amsxtra}
\usepackage{color}
\usepackage[]{subfigure}
\usepackage{cite}
\usepackage{hyperref}

\newcommand{\abs}[1]{\left |#1\right |}

\newcommand{\paren}[1]{\left(#1\right)}
\newcommand{\bracket}[1]{\left[#1\right]}
\newcommand{\set}[1]{\left\{#1\right\}}

\newcommand{\E}{\mathbb{E}}

\newcommand{\R}{\mathbb{R}}

\newtheorem{thm}{Theorem}
\newtheorem{prop}{Proposition}

\newtheorem{lem}{Lemma}

\newtheorem{asmp}{Assumption}

\newcommand{\dmd}{{\tiny \mathrm{DMD}}}
\newcommand{\mf}{{\tiny \mathrm{mf}}}

\newcommand{\Dkl}{\mathcal{R}}

\newcommand{\bx}{\mathbf{x}}
\newcommand{\bX}{\mathbf{X}}
\newcommand{\bc}{\mathbf{c}}
\newcommand{\ba}{\mathbf{a}}
\newcommand{\bk}{\mathbf{k}}
\newcommand{\bmu}{\boldsymbol{\mu}}

\newcommand{\Vdmd}{V_{\mathrm{\tiny DMD}}}
\newcommand{\Vhdmd}{\hat{V}_{\mathrm{\tiny DMD}}}

\newcommand{\nuAB}{\nu_{\mathrm{\tiny AB}}}
\newcommand{\nuhAB}{\hat{\nu}_{\mathrm{\tiny AB}}}
\newcommand{\VAB}{V_{\mathrm{\tiny AB}}}
\newcommand{\VhAB}{\hat{V}_{\mathrm{\tiny AB}}}
\newcommand{\ZAB}{Z_{\mathrm{\tiny AB}}}
\newcommand{\ZhAB}{\hat{Z}_{\mathrm{\tiny AB}}}
\newcommand{\FAB}{F_{\mathrm{\tiny AB}}}
\newcommand{\FhAB}{\hat{F}_{\mathrm{\tiny AB}}}

\DeclareMathOperator{\Cov}{Cov}

\numberwithin{equation}{section}

\title{A Theoretical Examination of Diffusive Molecular Dynamics}
\author{G. Simpson}
\email{simpson@math.drexel.edu}
\address{Department of Mathematics, Drexel University}
\author{M. Luskin}
\email{luskin@umn.edu}
\address{School of Mathematics, University of Minnesota}
\author{D.J. Srolovitz}
\email{srol@seas.upenn.edu}
\address{Department of Materials Science and Engineering, Department
  of Mechanical Engineering and Applied Mechanics, University of Pennsylvania}
\thanks{GS and ML were supported by US Department of Energy Award
  DE-SC0012733.}
\date{\today}

\begin{document}
\maketitle

\begin{abstract}

  Diffusive molecular dynamics is a novel model for materials with
  atomistic resolution that can reach diffusive time scales. The main ideas of diffusive molecular dynamics are to first
  minimize an approximate variational Gaussian free energy of the
  system with respect to the mean atomic coordinates (averaging over
  many vibrational periods), and to then to perform a diffusive step
  where atoms and vacancies (or two species in a binary alloy) flow on
  a diffusive time scale via  a master equation.  We present a
  mathematical framework for studying this algorithm based upon
  relative entropy, or Kullback-Leibler divergence.  This adds
  flexibility in how the algorithm is implemented and interpreted.  We then compare
  our formulation, relying on relative entropy and absolute continuity
  of measures, to existing formulations.  The main difference amongst
  the equations appears in a model for vacancy diffusion, where
  additional entropic terms appear in our development.
\end{abstract}
\section{Introduction}

One of the outstanding challenges in atomistic simulation of condensed
systems, such as solids, liquids, and glasses, is access to
experimentally meaningful length and time scales.  The spatial scales
amenable to direct molecular dynamics (MD) simulations have grown over
time; MD simulations on current large scale parallel computational
facilities now reach over $10^{12}$ atoms or approximately 1
$\mu$m$^3$, \cite{largestmd}.  Significantly larger length scales can
be achieved by application of multiscale modeling techniques that
combine atomistic simulations with continuum methods (see, for
example,
\cite{AbrahamBroughtonBernsteinEtAl1998,ShilkrotCurtinMiller2002,BlancLeBrisLegoll2005,RuddBroughton2005,WagnerLiu2003,TadmorOrtizPhillips1996,acta.atc}).
With regard to time, MD simulations have fundamental time scales
associated with atomic vibration periods ($\sim10^{-13}$ s), but
typical MD time steps are two orders of magnitude smaller.  The
longest times that have been achieved in large scale MD simulations on
special purpose hardware is $\sim10^{-3}$ s, \cite{longestmd}.  More
typically MD simulations access times of less than than $10^{-8}$ s.
Hence, reaching laboratory time scales remains amongst the most
important challenges in the application of MD today.

Many approaches have been developed to address the molecular dynamics
time-scale challenge.  Of these, methods based upon transition state
theory have been widely applied and their continued development
remains an active area of research.  Transition state theory-based
methods rely on rare event ideas in which the system explores a
particular energy basin for a long period of time and makes infrequent
transitions from one basin to another,
\cite{Voter:1997p12731,Voter:1998p13729,Henkelman:1999is,LeBris:2012et,Simpson:2013cs}.
Access to long time scales is provided by replacing the true dynamics
within the basin with a stochastic approximation and the prediction of
the times between the transitions to other basins.  Such methods rely
on the transitions being sufficiently rare; i.e., the computational
time required to characterize the dynamics within the basin is much
shorter than the time between inter-basin transits.  The main
difficulty with this approach arises when the transitions between
basins are not rare; for example, at high temperature or in situations
where there are low energy barriers between basins.  Another challenge
arises in very large systems as the time between rare events occurring
somewhere in the system typically scales inversely with system size.
Nonetheless, progress has recently been made to extend these methods
to large systems, ~\cite{hyperqc,Kim:2013ee,Xu:2008jk}.

An important distinction between classes of ``events'' in many
materials science applications is between those that are displacive
and those that are diffusive.  Displacive events are often collective
and involve relative atomic displacements that are small compared with
a typical interatomic separation ($\sim2 \textrm{\AA}$).  Diffusive
events, on the other hand, often involve a series of atom or vacancy
hops amongst atomic sites.  In a solid, the time between hops is
commonly many orders of magnitude larger than the atomic vibration
period.  This type of scale separation is necessary for the
application of traditional transition state theory-based approaches.
The motion of defects in materials and many types of phase
transformations occur through a combination of displacive and
diffusive events.

\subsection{Diffusive Molecular Dynamics}

One method that explicitly takes advantage of this separation in scale
between diffusive and displacive events in atomistic simulations is
the so-called Diffusive Molecular Dynamics (DMD) method
\cite{Li:2011gn,Sarkar:2012ce}.  The idea is to first minimize an
approximate free energy (including atomic bonding and atomic
vibrational degrees of freedom) of the system with respect to the mean
atomic coordinates (averaging over many vibrational periods), and then
to perform a diffusive step where atoms and vacancies flow on a
diffusive time scale.  This free energy is typically described using
the variational Gaussian (VG) method
\cite{LeSar:1989wz,Najafabadi:1990uc}.  The time step for such
simulations is the relatively long (compared with vibrational)
diffusional time scale.  In this approach, the search for transition
state barriers is replaced by the introduction of an empirical
diffusion coefficient or mobility.  While this coefficient may be
determined directly from atomistic simulation for every local
transition, it is commonly viewed as a constant across the entire
simulation; this constant can also be determined from purely atomistic
calculations.  This approach is physically sensible when most of the
important diffusive events are of the same type, as in an exchange of
an atom for a vacancy.

A key element of DMD is that, in contrast to traditional MD, each
atomic site has associated with it a continuous probability of
occupancy by an atom. See, for example,
\cite{Najafabadi_Wang_Srolovitz_LeSar_1991}.  We denote this
probability $c_i$ at atomic site $i$.  In the case of a single species
(elemental) materials, $c_i$ close to zero corresponds to site $i$
being a vacancy with high probability, while $c_i$ close to one
corresponds to site $i$ being occupied by an atom with high
probability.  Alloys can also be described in this manner.  For
example, $c_i$ close to zero could correspond to species A while $c_i$
close to one could correspond to species B.  This can be extended to
multiple species, along with vacancies, by the introductions of
additional degrees of freedom.  Such an extension of the state space,
to allow for longer time scale evolution, has also been explored in
\cite{Venturini:2014fv}.

\subsection{Relative Entropy}

One of the tools we make use of in this analysis is relative entropy,
or the Kullback-Leibler divergence (KL), \cite{dupuis1997weak}.  KL is
one of the many ways that the distance between two probability
measures, such as ensembles, can be computed.  It is broadly used in
information theory, uncertainty quantification, statistical inverse
problems, molecular dynamics, and other applications; see, for
instance,
\cite{Katsoulakis:2013cq,Chaimovich:2010ic,Shell:2008cj,Majda:2002wt,Pinski:2013vq,Pinski:2014ul}
and references therein.

Given two probability measures, $\nu$ and $\mu$ on a common state
space, the relative entropy distance between them is given by
\begin{equation}
  \label{e:Dkl1}
  \Dkl({\mu}||\nu) = \begin{cases} \E^{{\mu}}\bracket{\log \frac{d{\mu}}{d\nu}}&
    {\mu} \ll \nu,\\
    \infty & \text{otherwise.}
  \end{cases}
\end{equation}
Here, $\mu\ll \nu$ if for any measurable set $A$ such that $\nu(A)=0$,
we have $\mu(A)=0$ too.  In this case, $\mu$ is said to be absolutely
continuous with respect to $\nu$.

Relative entropy is a natural tool for this work as it can be directly
connected to the Helmholtz free energy of a system.  Indeed, the
statement that $\Dkl\geq 0$ is equivalent to the Gibbs-Bogoliubov
inequality~\cite{Shell:2008cj}.  To motivate this, consider the
canonical ensembles associated with potential $V(x)$ and an
approximate potential $\hat{V}(x)$:
\begin{equation}
  \nu(dx) = Z^{-1} e^{-\beta V(x)}dx, \quad \hat\nu(dx) = \hat{Z}^{-1}
  E^{-\beta \hat{V}(x)}dx,
\end{equation}
where $Z$ is the partition function, $\beta=(k_{\rm B}T)^{-1}$,
$k_{\rm B}$ is the Boltzmann constant and $T$ is the absolute
temperature.  Assuming that $\hat{\nu}\ll \nu$,
\begin{equation}
  \Dkl(\hat{\nu}||\nu) = \beta\E^{\hat{\nu}}[V(x) - \hat{V}(x)] - \log
  \hat{Z} + \log Z.
\end{equation}
Dividing through by $\beta$, and using $\Dkl \geq 0$,
\begin{equation}
  \label{e:gb_inequality}
  -\beta^{-1} \log Z \leq \beta\E^{\hat{\nu}}[\Delta V] - \beta^{-1}\log \hat{Z},
\end{equation}
which is the Gibbs-Bogoliubov inequality.  Mathematical treatments of
$\Dkl$ and its relationship to other popular measures of distance
between probability measures can be found in \cite{dupuis1997weak,
  Gibbs:2002fk}.

\subsection{Outline}

In this work, we formulate DMD using relative entropy, towards the
goal of constructing a rigorous mathematical foundation for the
algorithm.  We also examine under what assumptions the necessary
absolute continuity will hold.  This framework leads to naturally well
posed variational problems that are inherently bounded from below by
virtue of $\Dkl$ being non-negative.

Our paper is outlined as follows.  In Section \ref{s:kl_dmd} we give
our formulation of DMD using relative entropy.  Then, in Section
\ref{s:discussion}, we compare our expressions with those of the
existing formulations of DMD.  Several additional explicit
calculations are given in the Appendix.

\section{A Relative Entropy Formulation of DMD}
\label{s:kl_dmd}

To begin our description of DMD, we work in an extended state space
that includes both the positions of the atom sites, $\bx_i\in \R^3$, and
the presence/absence of atoms, $a_i \in \set{0,1}$.  For binary
alloys, $a_i=1$ could correspond to species A and $a_i = 0$ could
correspond to species B.  In either case, the total number of sites
$N$ is fixed and $i=1,\ldots N$.  In what follows, we give separate
developments of the vacancy and the binary alloy problems.

\subsection{DMD Potentials}
We assume the existence of a potential $V$, that describes atomic
bonding, as in traditional MD.  Here, we focus on the class of pair
potentials,
\begin{equation}
  \label{e:pair}
  V(\bx)  = \sum_{i<j} \phi (\abs{\bx_i - \bx_j}).
\end{equation}
The associated DMD potential is
\begin{equation}
  \label{e:Vdmd}
  \Vdmd(\bx,\ba)  = \sum_{i<j} a_i a_j \phi (\abs{\bx_i - \bx_j}).
\end{equation}
In an elemental material, the potentials $V$ and $\Vdmd$ are non-zero
only when both atom $i$ and atom $j$ are present (i.e., the
interaction between an atom and a vacancy is zero).  In the case of a
binary alloy (ignoring vacancies), we would have
\begin{equation}
  \label{e:Vdmd_binary}
  \begin{split}
    V_{AB}(\bx,\ba) & = \sum_{i<j} a_i a_j \phi_{AA} (\abs{\bx_i - \bx_j}) \\
    &\quad +\sum_{i<j} [ a_i
    (1-a_j) + (1-a_i)a_j ]\phi_{AB} (\abs{\bx_i - \bx_j})\\
    &\quad + \sum_{i<j} (1-a_i)(1-a_j) \phi_{BB} (\abs{\bx_i -
      \bx_j}),
  \end{split}
\end{equation}
where $\phi_{T_iT_j}$ is the pair potential between a pair of atoms of
types $T_i$ and $T_j$ $\in \set{A,B}$. More sophisticated potentials,
such as the embedded atom method (EAM), can be used, but for the
purpose of this work, it will suffice to consider the pair potential.

\subsection{Free Energy and Relative Entropy for the Vacancy Problem}

\subsubsection{Canonical Ensemble}
Letting $D$ be a {bounded, open} subset of $\R^3$, we now formulate
the canonical ensemble for a ``$\bc NT$'' (fixed total composition,
number of atomic sites and temperature) system, by introducing the
distribution
\begin{equation}
  \label{e:nu}
  \nu(d\bx, \ba) = Z^{-1} \exp\set{-\beta (\Vdmd - \bmu \cdot \ba)}d\bx,
\end{equation}
with $\bmu = (\mu_1, \ldots, \mu_N)$ is a set of chemical potentials.
Chemical potential $\mu_i$ is associated with site $i$; each site has
its own reservoir.  While this is in the form of a generalized grand canonical
ensemble, we view the chemical potentials as Lagrange multipliers
enforcing the constraints
\begin{equation}
  \label{e:c_constraint}
  \E^{\nu}[a_i] = c_i,\quad i = 1, \ldots N.
\end{equation}
It is in this way that our $\nu$ is akin to a canonical ensemble, with
the mean occupancy values, the $\bc$, fixed.


The partition function is
\begin{equation}
  \label{e:partition}
  Z = \sum_{\ba} \int_D \exp\set{-\beta (\Vdmd(\bx,\ba) - \bmu \cdot \ba)}d\bx,
\end{equation}
and the ensemble averages require integration in space over $D$ along
with summation over all possible configurations of
$\ba = (a_1, \ldots, a_N)$.  We thus assume:
\begin{asmp}
  \label{a:Zmu}
  For potential $V$, bounded set $D$, and occupancies $\bc$ with
  $c_i \in (0,1)$, the chemical potentials $\mu_i$ are finite and the
  partition function is finite and positive.
\end{asmp}

We describe the distribution \eqref{e:nu} as the ``true'' DMD
distribution.  The associated free energy is
\begin{equation}
  \label{e:DMDfree}
  F = -\beta^{-1} \log Z + \bmu \cdot \bc.
\end{equation}
In DMD, the concentrations are allowed to evolve under
\begin{equation}
  \label{e:DMDdynamics}
  \dot c_i = \sum_{j\in N(i)} k_{ij} \paren{\frac{\partial F}{\partial
      c_j} - \frac{\partial F}{\partial
      c_i}}
\end{equation}
where $k_{ij}$ is a mobility term, and $N(i)$ includes neighbors of
site $i$; both are, {\it a priori}, unconstrained. We note that
\eqref{e:DMDdynamics} is not the evolution studied by
Sarkar~\cite{Sarkar:2011wu}, due to the challenge of
computing the partition function; rather, they employ an approximate
free energy.  Nevertheless, we contend that, were computational
complexity not an obstacle, \eqref{e:DMDdynamics} is the desired form
of dynamics.

As the $c_i$ evolve, the equilibrium positions of the atomic sites
will also change, and can be obtained by computing
\begin{equation}
  \label{e:meanX}
  \E^{\nu}[\bx_i].
\end{equation}

\subsubsection{Approximate Potential}
Due to the computational challenge in evaluating the partition
function, \eqref{e:partition}, an approximate potential is introduced,
$\hat{V}$, with a corresponding probability distribution, $\hat{\nu}$.
This potential is then tuned to provide a ``best'' approximation of
$\nu$ by $\hat{\nu}$.  Here, we consider the following quadratic
(harmonic) potential in the case of an elemental material (including
vacancies)
\begin{equation}
  \label{e:Vdmdapprox}
  \Vhdmd(\bx,\ba; \bk, \bX)  = \sum_{i} \frac{a_i k_i}{2} \abs{\bx_i -  \bX_i}^2.
\end{equation}
$\bX$ and $\bk$ are parameters approximating the mean atomic position
and its characteristic fluctuation.  The approximate probability
distribution is now
\begin{equation}
  \label{e:nu_approx}
  \hat{\nu}(d\bx, \ba; \bk, \bX) =\hat{Z}^{-1} \exp\set{-\beta(\Vhdmd  - \hat{\bmu}\cdot \ba)}d\bx.
\end{equation}
Again, the $\hat{\mu}_i$ are Lagrange multipliers which must satisfy
the analog of \eqref{e:c_constraint},
\begin{equation}
  \label{e:c_constraint_approx}
  \E^{\hat{\nu}}[a_i] = c_i .
\end{equation}
For this case, we can provide expressions for both the partition
function and the $\hat{\mu}_i$.  The approximate partition function is
\begin{equation}
  \label{e:partition_approx}
  \hat{Z} = \Pi_{i=1}^N \paren{|D|+ e^{\beta\hat\mu_i} Z_i}, \quad Z_i  \equiv \int_D\exp\set{-\tfrac{\beta
      k_i}{2}\abs{\bx_i - \bX_i}^2}d\bx_i,
\end{equation}
where $|D|$ is the volume of the computational domain.
The chemical potentials $\hat{\mu}_i$, defined by
\eqref{e:c_constraint_approx}, satisfy (see \eqref{chempot})
\begin{equation}
  \begin{split}
    \E^{\hat\nu}[a_i]=\frac{e^{\beta\hat\mu_i} Z_i}{|D| +
      e^{\beta\hat\mu_i} Z_i }=c_i.
  \end{split}
\end{equation}
These can then be solved for each $i$ to obtain:
\begin{equation}
  \label{e:mu_approx_vals}
  \hat{\mu}_i =\beta^{-1} \log \paren{\frac{c_i}{1-c_i}\cdot \frac{|D|}{Z_i}}.
\end{equation}
The expression for the chemical potentials, \eqref{e:mu_approx_vals},
can also be used to write a simplified expression for the partition
function:
\begin{equation}
\label{e:partition_approx2}
\hat{Z} = |D|^N\Pi_{i=1}^N \frac{1}{1-c_i}.
\end{equation}

The approximate free energy is now given by
\begin{equation}
  \label{e:free_approx}
  \begin{split}
    -\beta^{-1} \log \hat Z + \hat\bmu\cdot \bc = \beta^{-1} &\sum_{i=1}^N c_i \log c_i +(1-c_i)\log(1-c_i) \\
    &\quad + (c_i-1) \log|D| -c_i \log Z_i.
  \end{split}
\end{equation}
{\bf Note that because $D$ is a {bounded} set, $\bX_i$ is not exactly
  the mean position of atomic site $i$.}  See Appendix \ref{s:details}
for details.

\subsubsection{Relative Entropy and the Generalized Variational Approach}
The basis for optimizing \eqref{e:Vdmdapprox} so that $\hat{\nu}$
provides the best match to $\nu$ is inspired by the Gibbs-Bogoliubov
inequality and the variational Gaussian approach
\cite{LeSar:1989wz,Najafabadi:1990uc}.  As alluded to in the
introduction, this is closely related to the relative entropy metric.

Let us assume that $V$, the primitive potential in this problem, is
bounded on the set $D$.\footnote{While this formally excludes
  potentials which diverge near the origin, such as the Lennard-Jones
  potential, such divergences are routinely avoided by choosing a
  cut-off at an appropriately small interatomic separation.}  In
particular, we assume
\begin{asmp}
  \label{a:Vbound}
  There exists a constant $C>0$ such that for all $\bx \in D^N$ and for
  all $\ba$,
  \begin{equation}
    \label{e:Vbound}
    |\Vdmd(\bx,\ba)|\leq C.
  \end{equation}
\end{asmp}

Under these assumptions, we have the necessary absolute continuity to
proceed with using $\Dkl$:
\begin{prop}
  \label{p:ac1}
  Subject to Assumptions \ref{a:Zmu} and \ref{a:Vbound},
  $\hat\nu\ll \nu$ with Radon-Nikodym derivative
  \begin{equation}
    \label{e:RND}
    \frac{d\hat{\nu}}{d\nu}(\bx, \ba) =\frac{Z}{\hat{Z}} \exp\set{\beta(\Vdmd - \Vhdmd+ \beta (\hat{\bmu} - \bmu)\cdot \ba}
  \end{equation}
  and
  \begin{equation}
    \label{e:Dkl2}
    \begin{split}
      \Dkl(\hat{\nu}||\nu) &= \beta \E^{\hat{\nu}}[\Delta V]+ \beta
      (\hat{\bmu} - \bmu)\cdot \E^{\hat{\nu}}[\ba] + \log Z - \log
      \hat{Z} < \infty.
    \end{split}
  \end{equation}
\end{prop}
\begin{proof}
  Let $A$ be any measurable subset of the state space,
  $(\set{0,1} \times {D})^N$, of the form
  \begin{equation}
\label{e:simple_set}
    A = \Pi_{i=1}^N (a_i \times B_i)
  \end{equation}
  where each $B_i$ is a Lebesgue measurable subset of $D\subset\R^d$.
  If $\nu(A)=0$, then
  \begin{equation*}
    0 = \int_{B_1\times B_2\ldots \times B_N}
    \exp\set{-\beta(\Vdmd(\bx,\ba) -
      \bmu \cdot \ba)}d\bx\geq e^{-\beta C - \beta \sum_i \abs{\mu_i}}\Pi_i |B_i|,
  \end{equation*}
  where $|B_i|$ is the Lebesgue measure of the set $B_i$ and $C$ is
  the constant from Assumption \ref{a:Vbound}.  Since the $\mu_i$ are
  assumed to be finite, we can thus infer that the only way $\nu(A)=0$
  is if at least one of the $B_i$ has Lebesgue measure zero,
  regardless of the particular site occupancy values, $a_i$.  Then,
  since the $k_i$ defining $\Vhdmd$ are finite,
  \begin{equation}
    \begin{split}
      \hat\nu(A)& = \hat{Z}^{-1}\int_{B_1\times B_2\ldots \times B_N}
      \exp\set{-\beta(\Vhdmd(\bx,\ba) -
        \hat\bmu \cdot \ba)}d\bx\\
      &\leq \hat{Z}^{-1} e^{\beta \sum_i \abs{\hat\mu_i}} \Pi_i
      |B_i|=0.
    \end{split}
  \end{equation}
  Since this holds for all sets of type \eqref{e:simple_set}, we
  conclude it that for all measurable subsets of  $(\set{0,1}
  \times D)^N$ for which $\nu(A) =0$ we must have $\hat\nu(A) = 0$.
  This gives us absolute continuity of the measures.

  The Radon-Nikodym derivative is then given by \eqref{e:RND} and
  $\Dkl$ is given by \eqref{e:Dkl2}.  By our assumptions, $\Dkl$ will
  be finite.
\end{proof}

Since $\Dkl\geq 0$, we can use \eqref{e:c_constraint_approx} and
\eqref {e:free_approx} to express \eqref{e:Dkl2} as
\begin{equation}
  \begin{split}
    F \leq E^{\hat{\nu}}[\Delta V] + \beta^{-1} &\sum_{i=1}^N c_i
    \log
    c_i +  (1-c_i)\log(1-c_i)-c_i \log Z_i \\
    &\quad +(c_i-1)\log|D|.
  \end{split}
\end{equation}
We thus define the approximate DMD free energy, which is an upper
bound on the true DMD free energy, as
\begin{equation}
  \label{e:free_approx2}
  \begin{split}
    \hat{F} \equiv E^{\hat{\nu}}[\Delta V] +\beta^{-1}&\sum_{i=1}^N c_i
    \log c_i + (1-c_i)\log(1-c_i) -c_i \log Z_i\\
    &\quad +(c_i-1)\log|D|.
  \end{split}
\end{equation}
The DMD algorithm for the vacancy problem proceeds in two steps:
\begin{enumerate}
\item Find a minimizer of $\hat{F}$ over $\bk$ and $\bX$, with
  $\bX_i \in D$ and $k_i \geq 0$.
\item Approximate the dynamics of \eqref{e:DMDdynamics}, substituting
  $\hat{F}$ for $F$.
\end{enumerate}
Thus, since the Gibbs-Bogliubov inequality is equivalent to the
statement that the relative entropy is non-negative, the first step in
the above algorithm is to find the best approximation, with respect to
relative entropy, of $\nu$ over a class of distributions of type
$\hat{\nu}$.

\subsection{Free Energy and Relative Entropy for a Binary Mixture}

Mirroring our examination of the vacancy problem, we consider the
analogous formulation for a binary alloy.

\subsubsection{Canonical Ensemble}
For a binary mixture, many of the calculations are similar, or even
simpler, than for the case of an elemental material with vacancies.
First, we formulate the $\nuAB$ distribution:
\begin{equation}
  \label{e:nuAB}
  \nuAB(d\bx, \ba) = \ZAB^{-1} \exp\set{-\beta(\VAB - \bmu\cdot \ba)}d\bx.
\end{equation}
We continue to assume Assumption \ref{a:Zmu} holds, adapted to the
binary mixture case.

\subsubsection{Approximate Ensemble}
Now, instead of the approximate potential given by
\eqref{e:Vdmdapprox} for the vacancy problem, we assume
\begin{equation}
  \label{e:VhAB}
  \VhAB = \sum_i \frac{k_i}{2}\abs{\bx_i - \bX_i}^2.
\end{equation}
The distinction here is that because each site always contains an atom
of species A or B, the potential is always non-zero.  In contrast, the
potential associated with a vacancy is always zero and a vacancy is
never subject to a force.  Proceeding with the potential
\eqref{e:VhAB}, we find
\begin{equation}
  \nuhAB(d\bx, \ba; \bk, \bX) = \ZhAB^{-1} \exp\set{-\beta (\VhAB -
    \hat\bmu\cdot \ba ) }d\bx,
\end{equation}
where the chemical potentials are chosen to satisfy
$\E^{\nuhAB}[a_i]=c_i$.  In this case
\begin{equation}
  \label{e:Zh_AB}
  \ZhAB = \Pi_{i=1}^N \paren{1 + e^{\beta \hat\mu_i}} Z_i,
\end{equation}
where $Z_i$ defined as in \eqref{e:partition_approx}.  Hence, we can
immediately write the chemical potential as
\begin{equation}
  \label{e:muAB}
  \hat\mu_i = \beta^{-1} \log \paren{\frac{c_i}{1-c_i}}.
\end{equation}
The free energy is then given by
\begin{equation}
  \label{e:free_AB_approx}
  -\beta^{-1} \log \hat Z + \bmu \cdot \bc = \beta^{-1} \sum_{i=1}^N c_i
  \log c_i + (1-c_i) \log(1-c_i) - \log Z_i .
\end{equation}

\subsubsection{Relative Entropy for the Binary Mixture}
We next consider the relative entropy minimization problem for the
binary mixture, requiring $\nuhAB \ll \nuAB$.  Using the same approach
as in the vacancy case, we will also assume the boundedness of $\VAB$,
in the same spirit as Assumption \ref{a:Vbound}:
\begin{prop}
  Under Assumptions \ref{a:Zmu} and \ref{a:Vbound} for the binary
  mixture case, $\nuhAB \ll \nuAB$ with the associated Radon-Nikodym
  derivative, and
  \begin{equation}
    \label{e:DklAB}
    \Dkl(\nuhAB||\nuAB) = \beta \E^{\nuhAB}[\Delta V] + \beta (\hat\bmu- \bmu)
    \cdot \bc - \log \ZhAB + \log \ZAB<\infty.
  \end{equation}
\end{prop}

As before, we can reformulate this as a free energy statement,
\begin{equation}
  \begin{split}
    &\underbrace{-\beta^{-1}\log \ZAB + \bmu \cdot \bc}_{\FAB} \\
    &\leq\underbrace{\E^{\nuhAB}[\Delta V] + \beta^{-1} \sum_{i=1}^N
      c_i \log c_i + (1-c_i) \log(1-c_i) - \log Z_i}_{\FhAB}.
  \end{split}
\end{equation}
The simulation now proceeds as above, with a minimizer $\Dkl$ over
$\bk$ and $\bX$ while the $c_i$ evolve.

\section{Discussion}
\label{s:discussion}

\subsection{Relation to Existing DMD Formulations}
\subsubsection{Vacancy Problem}
We compare \eqref{e:free_approx2} to the DMD free energy for the
vacancy problem found in \cite{Li:2011gn,Sarkar:2011wu,Sarkar:2012ce}.
Indeed, if one rewrites equation (3) from the original 2011 DMD paper
\cite{Li:2011gn} for the case of the pair potential (in the classical
approximation -- {\it i.e.}, for Planck's constant equals zero), the
expression is:
\begin{tiny}
  \begin{equation}
    \label{e:free_2011}
    \begin{split}
      F_{2011} &= \sum_{i< j} c_i c_j \paren{\tfrac{2\pi}{\beta
          k_i}}^{\frac d 2} \paren{\tfrac{2\pi}{\beta k_j}}^{\frac d
        2}\iint_{\R^{2d}} \phi(|\bx_i - \bx_j|) e^{-\frac{\beta
          k_i}{2}|\bx_i - \bX_i|^2} e^{-\frac{\beta k_i}{2}|\bx_j -
        \bX_j|^2} d\bx_i d\bx_j\\
      & \quad+ \beta^{-1}\sum_{i=1}^N \frac{d}{2}
      c_i\paren{\log\tfrac{\beta k_i}{2\pi} -1} +c_i\log c_i +
      (1-c_i)\log(1-c_i).
    \end{split}
  \end{equation}
\end{tiny}
Some of the differences between \eqref{e:free_2011} and
\eqref{e:free_approx2} can be reconciled by the use of a {\it mean
  field} approximation, by which \eqref{e:Vdmd} is replaced with
\begin{equation}
  \label{e:Vdmd_mf}
  V_{\dmd, \mf} = \sum_{i < j} c_i c_j \phi(|\bx _i - \bx_j|)
\end{equation}
where the $c_i$ take continuous values between $0$ and
  $1.$ This approximation simplifies some of the computations.  For
instance, the $\mu_i$ are now explicit:
\begin{equation}
  \label{e:mu_mf_vals}
  \mu_i = \beta^{-1} \log\paren{ \frac{c_i}{1-c_i}},
\end{equation}
and the numerical estimation of $\E^{\hat \nu}[\Vdmd]$ is simplified.
Notice that \eqref{e:mu_mf_vals} is the same quantity as we obtained
in our examination of the binary mixture, \eqref{e:muAB}.  { This will
  hold generically when the potential for our distribution, whether
  true or approximate, does not explicitly depend on $\ba$.}
Furthermore, in the mean field case, the finiteness of the partition
function, assumed in Assumption \ref{a:Zmu}, is implied by boundedness
of the mean field potential over the set $D$ for all $\bc\in [0,1]^N$;
thus, an assumption like Assumption \ref{a:Vbound} is required.

However, the mean field approximation does not fully account for the
differences.  Part of the discrepancy may be attributed to the choice
of the state space; Sarkar et al. \cite{Sarkar:2012ce,Li:2011gn} use
$(\set{0,1}\times \R^d)^N$.  Formally, as $D \to \R^d$,
\begin{equation}
  Z_i \to \paren{\frac{2\pi}{\beta k_i}}^{\frac d 2}, \quad
  \E^{\hat\nu}[\Vhdmd] \to \beta^{-1}\frac{d}{2}\sum_{j=1}^N c_j,
\end{equation}
and \eqref{e:free_approx2} tends to
\begin{equation}
  \hat{F} \to F_{2011} + \beta^{-1}\sum_{i=1}^N (c_i-1) \log\abs{D},
\end{equation}
where we have not taken $\abs{D}$ to the limit in the last expression.
This dependence on $\abs{D}$ is a finite size correction to the
free energy.  As $D\to \R^d$, this becomes unbounded.  However, this
term does not alter the algorithm since we are only concerned with differences of
free energies rather than absolute magnitudes.  Indeed, during the
first part of the algorithm, where a local minimizer of $  \hat{F}$ is
sought over $\bk$ and $\bX$ with fixed $\bc$, there are no
contributions $\propto \log\abs{D}$.  And in the second part of the
algorithm, under the dynamics of type \eqref{e:DMDdynamics}, where
$ \hat{F}$ is used in place of $F$, the $\log\abs{D}$ terms cancel
one another.

The $\log\abs{D}$ terms also cancel under the more
sophisticated dynamics used in
\cite{Dontsova:2014et,Li:2011gn,Sarkar:2012ce}. There, the dynamics
are given by the master equation
\begin{equation}
  \label{e:dynamics2}
  \dot c_i = \sum_{j\in N(i)} \nu e^{-\beta Q_{\rm m}}\set{c_j (1-c_i)
    e^{-\beta (f_i - f_j)} - c_i (1-c_j)e^{-\beta (f_j - f_i)}},
\end{equation}
with
\begin{equation}
  \label{e:fdef}
  f_i \equiv \frac{\partial \hat{F}}{\partial c_i}-\beta^{-1}\log\paren{\frac{c_i}{1-c_i}}.
\end{equation}
Because of the differentiation with respect to $c_i$ in the previous
expression, the $\log|D|$ term does not appear in the differences,
$f_i - f_j$.

A more substantive difference between the existing DMD literature and our
analysis is the notion of the ensemble averaged site positions.  As
derived in the Appendix \eqref{e:mean_pos},
\begin{equation}
  \label{e:mean_pos2}
  \begin{split}
    \E^{\hat\nu}[\bx_i] &= \frac{\int_{D} \bx_i d\bx_i+
      e^{\beta\hat\mu_i} \int_D \bx_i\exp\set{-\tfrac{\beta
          k_i}{2}\abs{\bx_i - \bX_i}^2}d \bx_i}{|D| +
      e^{\beta\hat\mu_i} \int_D \exp\set{-\tfrac{\beta
          k_i}{2}\abs{\bx_i - \bX_i}^2}d \bx_i}\\
    & = (1-c_i) \frac{\int_{D} \bx_i d\bx_i}{|D|} + c_i \frac{\int_D
      {\bx_i}\exp\set{-\tfrac{\beta k_i}{2}\abs{\bx_i - \bX_i}^2}d
      \bx_i}{Z_i},
  \end{split}
\end{equation}
where we have used \eqref{e:mu_approx_vals} to simplify the
expression.  For simulations in computational domains that are
symmetric about the origin, such as $D=(-L,L)^d$ or a hypersphere, the first term in
this expression vanishes and $\E^{\hat\nu}[\bx_i] \to c_i \bX_i$.

An alternative notion of mean position could be useful in this case.
Consider:
\begin{equation}
  \label{e:mean_alt}
  \frac{\E^{\hat\nu}[a_i\bx_i]}{\E^{\hat\nu}[a_i]} = \frac{1}{Z_i}{\int_D {\bx_i}\exp\set{-\tfrac{\beta
        k_i}{2}\abs{\bx_i - \bX_i}^2}d \bx_i},
\end{equation}
where we have made use of \eqref{e:mean_weight_pos}.  Now, as $D$
tends to $\R^d$, \eqref{e:mean_alt} recovers $\bX_i$.  This weighted
averaging is inspired by the mathematical theory of multiphase flow
(see, for instance, \cite{drew1999tmf}).

Working with $\R^d$ is inherently problematic, as the measures defined
as in \eqref{e:nu} and \eqref{e:nu_approx} do not lead to well defined
probability measures, even under the mean field approximation.
Consider, for example, $\hat\nu$ with $N=1$.  For this problem the
partition function would be
\begin{equation}
  \sum_{a_1=0}^1 \int_{\R^d }\exp\set{-\beta\frac{a_1k_1}{2}|\bx_i -
    \bX_i|^2+\beta\hat \mu_1 a_1}d\bx_1 = \infty,
\end{equation}
since, in the case $a_1=0$, we are integrating $d\bx_1$ over the whole
space.

However, we contend that while \eqref{e:free_2011} may give rise to
physically consistent simulations, it is {\bf not} based on a
variational principle, and instead, practitioners should use
\eqref{e:nu_approx}.  In order to make use of $\Dkl$ and arrive at a
variational formulation, it is essential that the distributions be
well defined probability distributions, with finite partition
functions and that absolute continuity holds.  We note that this is
{not} an artifact of our mathematical analysis based on relative
entropy.  As the Gibbs-Bogliubov inequality is a restatement of the
non-negativity of relative entropy, it has the same underlying
assumptions of absolute continuity and well defined measures.

One way to correct \eqref{e:mean_pos2} is to alter the choice of
\eqref{e:Vdmdapprox}.  Suppose we mimic what is done in the binary
mixture (which does not suffer from these problems), and took
\begin{equation}
  \label{e:Vdmdapprox2}
  \Vhdmd = \sum_{i} \frac{k_i}{2}\abs{\bx_i - \bX_i}^2.
\end{equation}
We would replicate \eqref{e:Zh_AB} for the partition function and,
with this revised value, the free energy would be
\begin{equation}
  \hat{F}_{\dmd}= \E^{\hat\nu}[\Delta V] + \beta^{-1} \sum_{i=1}^N c_i \log c_i
  + (1-c_i)\log(1-c_i) - \log Z_i .
\end{equation}
Now, as $D\to \R^d$, we obtain
\begin{equation}
  \begin{split}
    \label{e:Fdmd_approx2}
    \hat{F}_{\dmd}\approx \E^{\hat\nu}[\Vdmd] + \beta^{-1}
    &\sum_{i=1}^N c_i \log c_i
    + (1-c_i)\log(1-c_i)  \\
    &\quad+ \frac{d}{2}\paren{\log\frac{\beta k_i}{2\pi}-1}.
  \end{split}
\end{equation}
This formulation has the advantage that
$\E^{\hat \nu}[\bx_i]\to \bX_i$ as $D\to \R^d$.  However, notice now
that there is no $c_i$ multiplying the last expression in
\eqref{e:Fdmd_approx2}, as in \eqref{e:free_2011}.

\subsubsection{Binary Mixture Problem}
As noted, if we make a mean field approximation, then as $D\to \R^d$,
we recover the expression found in \cite{Dontsova:2014et}, which we do
not reproduce here.

\subsection{Advantages of Relative Entropy}

One of the main advantages of the formulation given here is that the
question of relative entropy minimization is a rigorously defined
variational problem, forcing us to confront problems such as that
associated with domain size.  Indeed, for DMD we have
\begin{thm}
  Given $K\in (0, \infty)$ and the open bounded subset $D$ of $\R^d$,
  let
  \begin{equation}
    \label{e:admissible}
    \mathcal{A}_{D,K} = \set{\text{$\gamma$ of the form \eqref{e:nu_approx}}\mid
      \bX_i \in \bar D, \quad k_i \in [0, K]}.
  \end{equation}
  Then if $\set{\gamma_n}$ is a minimizing sequence of
  $\Dkl(\cdot||\nu)$, it has a subsequential limit such that
  \begin{equation}
    \label{e:minimizing_seq}
    \lim_{n\to \infty}\Dkl(\gamma_n||\nu) = \Dkl(\gamma_\star||\nu) =
    \inf_{\gamma\in \mathcal{A}} \Dkl(\gamma||\nu)
  \end{equation}
  and
  \begin{equation}
    \label{e:tv_conv}
    \gamma_n \overset{\mathrm{TV}}{\to} \gamma_\star.
  \end{equation}
\end{thm}
Note that by our previous assumptions on $V$ and $\bmu$, that they are
bounded on $D$ for $c_i\in (0,1)$, we are assured that the elements of
$\mathcal{A}_{D,K} $ are absolutely continuous with respect to such a
$\nu$, and have finite $D_{kl}$.
\begin{proof}
  We claim $\mathcal{A}_{D,K}$ is weakly compact; see Appendix
  \ref{a:weak}.  Therefore, any sequence has a weak limit,
  $\gamma_{n_k}\overset{\mathrm{w}}{\to} \gamma_\star$.  Then, by the
  lower semicontinuity of $\Dkl(\cdot||\nu)$ (Proposition 2.1 of
  \cite{Pinski:2013vq}) we have \eqref{e:minimizing_seq}, then we can
  infer \eqref{e:tv_conv} (Lemma 2.4 of \cite{Pinski:2013vq}).
\end{proof}
Thus, the free energy minimization part of the algorithm is well posed
in the sense that if we take a minimizing sequence it has a
subsequential limit.

Another advantage of this formulation of the problem is that it allows
one to consider more general parameterizations of the synthetic
distribution, $\hat \nu$.  Indeed, one could imagine any distribution
parameterized by some collection of variables, denoted collectively by
$\mathbf{p}$, and proceed as above; first, one finds a local minimizer
of $\Dkl(\hat\nu_{\bf p} || \nu)$.  Using that, the free energy
gradients with respect to $\bc$ are computed at this stationary point,
and this drives the dynamic evolution of the $c_i$'s.  Provided this
admissible class is closed (in some sense), we are ensured that its
minimization problem is well posed too.

\subsection{Open Problems and Future Work}

Here, we have presented a mathematical framework for DMD, and
we have given particular attention to the question of free energy
minimization.  One point we have not addressed is the temporal
evolution problem, and how the master equation will be influenced by
the use of the approximate free energy in place of the true DMD free
energy. Indeed, we note that while the minimization of free energy
ensures that $\hat{\nu}$ is close to $\nu$ in the sense of relative
entropy (and hence, total variation), we cannot conclude that
\begin{equation}
  \abs{\frac{\partial F}{\partial c_i}-\frac{\partial   \hat{F}}{\partial c_i}},
\end{equation}
errors in the dynamics of both \eqref{e:DMDdynamics} and
\eqref{e:dynamics2} are small.  As of now, this remains an
unconstrained approximation that merits investigation.  It can be
shown that
\begin{equation}
  \label{e:dFdc}
  \frac{\partial F}{\partial c_i} = \mu_i,
\end{equation}
and, at a minimizer of $\Dkl$,
\begin{equation}
  \label{e:dFhdc}
  \frac{\partial   \hat{F}}{\partial c_i} = \hat\mu_i +
  \beta \Cov_{\hat\nu}\paren{\Delta V, \partial_{c_i} \hat\bmu \cdot \ba}.
\end{equation}
Calculations of these can be found in the appendix.  Similar
expressions hold under the mean field approximation.

The validity of the mean field approximation used in
\cite{Dontsova:2014et,Li:2011gn,Sarkar:2012ce} remains to be explored.
As noted, this dramatically simplifies the calculations and
immediately gives the ``true'' chemical potentials
\eqref{e:mu_mf_vals}.  The solvability in the general case is also an
open problem.

Finally, there is the question of how, or in what sense, does DMD, as
a model, approximate any specific, more primitive MD model, such as
Langevin dynamics.  More specifically, since DMD removes the
stochastic nature of primitive MD models, replacing the displacive
dynamics with a finite-temperature, energy minimization, DMD is
deterministic.  Similarly, the stochastic or random walk-like
diffusive dynamics in primitive MD becomes deterministic in DMD. The
question of when this is and is not acceptable remains to be explored.

\appendix
\section{Detailed Calculations}
\label{s:details}

\subsection{Partition Functions}
To obtain the approximate partition function
\eqref{e:partition_approx}, we employ the following procedure:
\begin{equation}
  \label{e:partition_calc1}
  \begin{split}
    \hat Z &= \sum_{\ba} \int_{D^N}
    \exp\set{-\beta(\hat{V}(\bx,\ba;\bk,\bX) - \hat\bmu \cdot \ba )} d\bx\\
    & = \Pi_{i=1}^N \set{\sum_{a_i} \int_D\exp\set{-\tfrac{\beta a_i
          k_i}{2}\abs{\bx_i - \bX_i}^2 + \beta \hat\mu_i a_i}d\bx_i}\\
    &= \Pi_{i=1}^N \set{|D|+
      e^{\hat\mu_i}{\int_D\exp\set{-\tfrac{\beta
            k_i}{2}\abs{\bx_i - \bX_i}^2}d\bx_i}}\\
    & = \Pi_{i=1}^N \set{|D|+ e^{\hat\mu_i} Z_i}.
  \end{split}
\end{equation}

\subsection{Chemical Potentials}
To obtain \eqref{e:mu_approx_vals}, we apply \eqref{e:partition_calc1}
to obtain
\begin{equation}\label{chempot}
  \begin{split}
    \E^{\hat\nu}[a_i]& = \frac{\Pi_{j\neq i}^N \set{|D|+
        e^{\beta\hat\mu_j} Z_i}}{\hat Z}\set{\sum_{a_i} \int_D
      a_i\exp\set{-\tfrac{\beta a_i
          k_i}{2}\abs{\bx_i - \bX_i}^2 + \beta \hat\mu_i a_i}d \bx_i}\\
    &=\frac{e^{\beta\hat\mu_i} Z_i}{|D| + e^{\beta\hat\mu_i} Z_i }.
  \end{split}
\end{equation}
Since $\E^{\hat\nu}[a_i] = c_i$, we solve for $\hat{\mu}_i$ in terms
of $c_i$.

\subsection{Mean Position}

The expectation value of the atomic position $\bx_i$ is
\begin{equation}
  \label{e:mean_pos}
  \begin{split}
    \E^{\hat\nu}[\bx_i]& = \frac{\Pi_{j\neq i}^N \set{|D|+
        e^{\beta\hat\mu_j} Z_i}}{\hat Z}\set{\sum_{a_i} \int_D
      \bx_i\exp\set{-\tfrac{\beta a_i
          k_i}{2}\abs{\bx_i - \bX_i}^2 + \beta \hat\mu_i a_i}d \bx_i}\\
    &=\frac{1}{|D| + e^{\beta\hat\mu_i} Z_i }\paren{\int_D \bx_i
      d\bx_i+ e^{\beta\hat\mu_i} \int_D \bx_i\exp\set{-\tfrac{\beta
          k_i}{2}\abs{\bx_i - \bX_i}^2}d \bx_i}
  \end{split}
\end{equation}
and the weighted mean position is given by
\begin{equation}
  \label{e:mean_weight_pos}
  \begin{split}
    \E^{\hat\nu}[a_i\bx_i]& = \frac{\Pi_{j\neq i}^N \set{|D|+
        e^{\beta\hat\mu_j} Z_i}}{\hat Z}\set{\sum_{a_i} \int_D
      a_i\bx_i\exp\set{-\tfrac{\beta a_i
          k_i}{2}\abs{\bx_i - \bX_i}^2 + \beta \hat\mu_i a_i}d \bx_i}\\
    &=\frac{1}{|D| + e^{\beta\hat\mu_i} Z_i }\paren{
      e^{\beta\hat\mu_i} \int_D \bx_i\exp\set{-\tfrac{\beta
          k_i}{2}\abs{\bx_i - \bX_i}^2}d \bx_i}.
  \end{split}
\end{equation}

\subsection{Mean Potential}

The expectation value of the mean potential $\Vhdmd$ is obtained as
follows:
\begin{equation}
  \label{e:mean_potential1}
  \begin{split}
    \E^{\hat\nu}[\Vhdmd] &=\sum_{j=1}^N \E^{\hat
      \nu}\bracket{\frac{a_j
        k_j}{2}\abs{\bx_j - \bX_j}^2}\\
    &= \sum_{j=1}^N \frac{\Pi_{k\neq j}^N \set{|D|+ e^{\beta\hat\mu_k}
        Z_k}}{\hat Z} e^{\beta \hat\mu_j }\int_D
    \frac{k_j}{2}\abs{\bx_j - \bX_j}^2 \exp\set{-\tfrac{\beta
        k_j}{2}\abs{\bx_j - \bX_j}^2} d\bx_j\\
    & = \sum_{j=1}^N \frac{c_jk_j}{2Z_j}\int_D\abs{\bx_j - \bX_j}^2
    \exp\set{-\tfrac{\beta k_j}{2}\abs{\bx_j - \bX_j}^2} d\bx_j .
  \end{split}
\end{equation}

\subsection{Free Energy Gradients}
The DMD free energy gradient, with respect to $c_i$, is computed as
follows
\begin{equation*}
  \begin{split}
    \frac{\partial F}{\partial c_i}& = - \beta^{-1} \frac{1}{Z}
    \frac{\partial Z}{\partial c_i} + \frac{\partial \bmu}{\partial
      c_i}\cdot \bc +
    \mu_i\\
    & = - \beta^{-1} \frac{1}{Z} \set{\sum \int \beta \frac{\partial
        \bmu}{\partial c_i}\cdot \ba\exp\set{-\beta \Vdmd + \bmu \cdot
        \ba} d\bx}+ \frac{\partial \bmu}{\partial c_i}\cdot \bc +
    \mu_i\\
    & = - \frac{\partial \bmu}{\partial c_i}\cdot\E^\nu[\ba]+
    \frac{\partial \bmu}{\partial c_i}\cdot \bc + \mu_i = \mu_i.
  \end{split}
\end{equation*}
This is \eqref{e:dFdc}.  The gradient of $\hat{F}$ with respect to any
of the parameters defining $\hat\nu$ vanishes at the minimizer of
$\Dkl$.  To see this, recall that since we have a minimizer of $\Dkl$,
\[
\frac{\partial\Dkl(\hat\nu||\nu)}{\partial p} = 0
\]
for any parameter, $p$, such as $k_i$ and $\bX_i$.  Since we can write
\begin{equation*}
  \beta^{-1} \Dkl = \hat{F}-F
\end{equation*}
and $F$ does not depend on the parameters, when a derivative is
taken with respect to $c_i$, chain rule terms involving the parameters
vanish at the minimizers.  Therefore,
\begin{equation*}
  \begin{split}
    \frac{\partial \hat{F}}{\partial c_i}& = - \beta^{-1} \frac{1}{\hat
      Z} \frac{\partial \hat Z}{\partial c_i} + \frac{\partial
      \hat\bmu}{\partial c_i}\cdot \bc +
    \hat\mu_i + \frac{\partial \E^{\hat\nu}[\Delta V]}{\partial c_i}\\
    & = \hat\mu_i + \frac{1}{\hat{Z}}\sum_{a_i}\int \Delta V
    \beta( \partial_{c_i} \hat\bmu \cdot \ba) \exp\set{-\beta \hat V
      +\beta
      \hat\bmu\cdot \ba}\\
    &\quad- \frac{1}{Z}\E^{\hat\nu}[\Delta V]\paren{\sum_{a_i}\int
      \beta( \partial_{c_i} \hat\bmu \cdot \ba) \exp\set{-\beta \hat V
        +\beta \hat\bmu\cdot \ba}},
  \end{split}
\end{equation*}
and this gives us \eqref{e:dFhdc}.

\section{Weak Compactness of the set of Measures}
\label{a:weak}
\begin{lem}
  The set $\mathcal{A}_{D,K}$ is weakly compact, in the sense that if
  $\set{\gamma_n}$ is any sequence in $\mathcal{A}_{D,K}$, it has a
  weakly converging subsequence in $\mathcal{A}_{D,K}$.
\end{lem}
\begin{proof}
  Given any sequence of $\gamma_n \in \mathcal{A}_{D,K}$, we have
  sequences $(\bX^{(n)}, \bk^{(n)}) \in \bar D^N\times [0,K]^N$.
  Since the set is compact, it has a convergent subsequence,
  \[
  \bX_{i}^{(n_m)} \to \bX_{i}^{(\star)},\quad k_{i}^{(n_m)} \to
  k_{i}^{(\star)}
  \]
  Let $\gamma_\star$ be the measure associated with $\bX^{(\star)}$
  and $\bk^{(\star)}$.  Clearly, $\gamma_\star \in \mathcal{A}_{D,K}$.

  Given any bounded continuous function $f$ on the set $X$, we will
  now show
  \[
  \E^{\gamma_{n_m}}[f] \to \E^{\gamma_\star}[f].
  \]
  From \eqref{e:partition_approx2},  so long as the
  $c_i \in (0,1)$ $Z_n = Z_\star$.  For brevity, let
  \begin{gather*}
    \Vhdmd(\bx, \ba;\bX^{(n_m)}, \bk^{(n_m)}) = \hat{V}^{(m)}(\bx,
    \ba)\to \hat{V}^{(\star)}(\bx,
    \ba)\\
    \hat\bmu(\bX^{(n_m)}, \bk^{(n_m)}) = \hat{\bmu}^{(m)}\to
    \hat{\bmu}^{(\star)}
  \end{gather*}

  Therefore, when we take differences,
  \begin{equation*}
    \begin{split}
      &\abs{\E^{\gamma_{n_m}}[f]- \E^{\gamma_\star}[f]
      }\\
      &\leq\frac{1}{Z_\star}\sum_{\ba} \int_{D^N} \abs{f(\bx, \ba)}
      \abs{e^{-\beta\hat{V}^{(m)}(\bx,\ba) + \beta \hat\bmu^{(m)}
          \cdot \ba }-e^{-\beta\hat{V}^{( \star)}(\bx,\ba) + \beta
          \hat\bmu^{(\star)}
          \cdot \ba   }}\\
      & \leq \frac{1}{Z_\star}\sum_{\ba} \int_{D^N} \abs{f(\bx,
        \ba)}e^{-\beta\hat{V}^{( \star)}(\bx,\ba) + \beta
        \hat\bmu^{(\star)} \cdot \ba
      }\abs{e^{-\beta(\hat{V}^{(m)}(\bx,\ba)-\hat{V}^{(
            \star)}(\bx,\ba)) + \beta \hat(\bmu^{(m)}-\bmu^{(\star)})
          \cdot \ba }-1}
    \end{split}
  \end{equation*}
  Since $\bx$ and $\ba$ are elements of bounded sets, and
  \[
  e^{-\beta(\hat{V}^{(m)}(\bx,\ba)-\hat{V}^{( \star)}(\bx,\ba)) +
    \beta \hat(\bmu^{(m)}-\bmu^{(\star)}) \cdot \ba }
  \]
  depends continuously upon $\bX^{(m)}$ and $\bk^{(m)}$, we then have
  there exists some constant such that for all $\bx\in \bar D$ and
  $\ba \in \set{0,1}$,
  \[
  \abs{e^{-\beta(\hat{V}^{(m)}(\bx,\ba)-\hat{V}^{( \star)}(\bx,\ba)) +
      \beta \hat(\bmu^{(m)}-\bmu^{(\star)}) \cdot \ba }-1}\leq
  C\paren{\abs{\bX^{(m)} - \bX^{(\star)}} + \abs{\bk^{(m)} -
      \bk^{(\star)}}}
  \]
  Therefore,
  \[
  \abs{\E^{\gamma_{n_m}}[f]- \E^{\gamma_\star}[f] }\leq
  C\paren{\abs{\bX^{(m)} - \bX^{(\star)}} + \abs{\bk^{(m)} -
      \bk^{(\star)}}}\E^{\gamma_\star}[|f|]
  \]
  and we have weak convergence.

\end{proof}

\bibliographystyle{plain}

\bibliography{dmd_refs}

\begin{thebibliography}{10}

\bibitem{AbrahamBroughtonBernsteinEtAl1998}
F.~F. Abraham, J.~Q. Broughton, N.~Bernstein, and E.~Kaxiras.
\newblock Spanning the length scales in dynamic simulation.
\newblock {\em Computers In Physics}, 12(6):538--546, 1998.

\bibitem{BlancLeBrisLegoll2005}
X.~Blanc, C.~Le~Bris, and F.~Legoll.
\newblock Analysis of a prototypical multiscale method coupling atomistic and
  continuum mechanics.
\newblock {\em M2AN. MathematicalModelling and Numerical Analysis},
  39(4):797--826, 2005.

\bibitem{Chaimovich:2010ic}
A.~Chaimovich and M.S. Shell.
\newblock {Relative entropy as a universal metric for multiscale errors}.
\newblock {\em Physical Review E}, 81(6):060104, June 2010.

\bibitem{Dontsova:2014et}
E.~Dontsova, J.~Rottler, and C.W. Sinclair.
\newblock {Solute-defect interactions in Al-Mg alloys from diffusive
  variational Gaussian calculations}.
\newblock {\em Physical Review B}, 90(17):174102, November 2014.

\bibitem{drew1999tmf}
D.A. Drew and S.L. Passman.
\newblock {\em {Theory of multicomponent fluids}}.
\newblock Springer New York, 1999.

\bibitem{dupuis1997weak}
P.~Dupuis and R.S. Ellis.
\newblock {\em A weak convergence approach to the theory of large deviations}.
\newblock Wiley Series in Probability and Statistics: Probability and
  Statistics. John Wiley \& Sons Inc., New York, 1997.
\newblock A Wiley-Interscience Publication.

\bibitem{largestmd}
Wolfgang Eckhardt, Alexander Heinecke, Reinhold Bader, Matthias Brehm, Nicolay
  Hammer, Herbert Huber, Hans-Georg Kleinhenz, Jadran Vrabec, Hans Hasse,
  Martin Horsch, Martin Bernreuther, ColinW. Glass, Christoph Niethammer, Arndt
  Bode, and Hans-Joachim Bungartz.
\newblock 591 {TFLOPS} multi-trillion particles simulation on {SuperMUC}.
\newblock In JulianMartin Kunkel, Thomas Ludwig, and HansWerner Meuer, editors,
  {\em Supercomputing}, volume 7905 of {\em Lecture Notes in Computer Science},
  pages 1--12. Springer Berlin Heidelberg, 2013.

\bibitem{Gibbs:2002fk}
A.L. Gibbs and F.E. Su.
\newblock {On Choosing and Bounding Probability Metrics}.
\newblock {\em International Statistical Review}, 70(3):419--435, December
  2002.

\bibitem{Henkelman:1999is}
G.~Henkelman and H.~J{\'o}nsson.
\newblock {A dimer method for finding saddle points on high dimensional
  potential surfaces using only first derivatives}.
\newblock {\em J. Chem. Phys.}, 111(15):7010--7022, 1999.

\bibitem{Katsoulakis:2013cq}
M.A. Katsoulakis and P.~Plech{\'a}{\v c}.
\newblock {Information-theoretic tools for parametrized coarse-graining of
  non-equilibrium extended systems}.
\newblock {\em The Journal of Chemical Physics}, 139(7):074115, 2013.

\bibitem{Kim:2013ee}
S.Y. Kim, D.~Perez, and A.F. Voter.
\newblock {Local hyperdynamics}.
\newblock {\em J. Chem. Phys.}, 139(14):144110, 2013.

\bibitem{hyperqc}
W.K. Kim, M.~Luskin, D.~Perez, E.~Tadmor, and A.F. Voter.
\newblock {Hyper-QC}: An accelerated finite-temperature quasicontinuum method
  using hyperdynamics.
\newblock {\em Journal of the Mechanics and Physics of Solids}, 63:94--112,
  2014.

\bibitem{LeBris:2012et}
C.~Le~Bris, T.~Leli{\`e}vre, M.~Luskin, and D.~Perez.
\newblock {A mathematical formalization of the parallel replica dynamics}.
\newblock {\em Monte Carlo Meth. Appl.}, 18(2):119--146, 2012.

\bibitem{LeSar:1989wz}
R.~LeSar, R.~Najafabadi, and D.J. Srolovitz.
\newblock {Finite-temperature defect properties from free-energy minimization}.
\newblock {\em Physical Review Letters}, 63(6):624, 1989.

\bibitem{Li:2011gn}
J.~Li, S.~Sarkar, W.T. Cox, T.J. Lenosky, E.~Bitzek, and Y.~Wang.
\newblock {Diffusive molecular dynamics and its application to nanoindentation
  and sintering}.
\newblock {\em Physical Review B}, 84(5):054103, August 2011.

\bibitem{acta.atc}
M.~Luskin and C.~Ortner.
\newblock Atomistic-to-continuum coupling.
\newblock {\em Acta Numerica}, 22:397--508, 2013.

\bibitem{Majda:2002wt}
A.~Majda, R.~Kleeman, and D.~Cai.
\newblock {A mathematical framework for quantifying predictability through
  relative entropy}.
\newblock {\em Methods and Applications of Analysis}, 9(3):425--444, 2002.

\bibitem{Najafabadi:1990uc}
R.~Najafabadi, D.J. Srolovitz, and R.~LeSar.
\newblock {Finite temperature structure and thermodynamics of the Au S5 (001)
  twist boundary}.
\newblock {\em Journal of Materials Research}, 5(11):2663, 1990.

\bibitem{Najafabadi_Wang_Srolovitz_LeSar_1991}
R.~Najafabadi, H.Y. Wang, D.J. Srolovitz, and R.~LeSar.
\newblock A new method for the simulation of alloys; application to interfacial
  segregation.
\newblock {\em Acta Metallurgica}, 39:12:3071--3082, 1991.

\bibitem{Pinski:2013vq}
F.J. Pinski, G.~Simpson, A.M. Stuart, and H.~Weber.
\newblock {Kullback-Leibler Approximation for Probability Measures on Infinite
  Dimensional Spaces}.
\newblock {\em arXiv.org}, October 2013.

\bibitem{Pinski:2014ul}
F.J. Pinski, G.~Simpson, A.M. Stuart, and H.~Weber.
\newblock {Algorithms for Kullback-Leibler Approximation of Probability
  Measures in Infinite Dimensions}.
\newblock {\em arXiv.org}, August 2014.

\bibitem{RuddBroughton2005}
R.E. Rudd and J.Q. Broughton.
\newblock Coarse-grained molecular dynamics: {N}onlinear finite elements and
  finite temperature.
\newblock {\em Physical Review B}, 72:144104, 2005.

\bibitem{Sarkar:2011wu}
S.~Sarkar.
\newblock {\em {Extending the Time Scale in Atomistic Simulations: The
  Diffusive Molecular Dynamics Method}}.
\newblock PhD thesis, Ohio State University, 2011.

\bibitem{Sarkar:2012ce}
S.~Sarkar, J.~Li, W.T. Cox, E.~Bitzek, T.J. Lenosky, and Y.~Wang.
\newblock {Finding activation pathway of coupled displacive-diffusional defect
  processes in atomistics: Dislocation climb in fcc copper}.
\newblock {\em Physical Review B}, 86(1):014115, July 2012.

\bibitem{Shell:2008cj}
M.S. Shell.
\newblock {The relative entropy is fundamental to multiscale and inverse
  thermodynamic problems}.
\newblock {\em The Journal of Chemical Physics}, 129(14):144108, 2008.

\bibitem{ShilkrotCurtinMiller2002}
L.~E. Shilkrot, W.~A. Curtin, and R.~E. Miller.
\newblock A coupled atomistic/continuum model of defects in solids.
\newblock {\em Journal of the Mechanics and Physics of Solids}, 50:2085--2106,
  2002.

\bibitem{Simpson:2013cs}
G.~Simpson and M.~Luskin.
\newblock {Numerical analysis of parallel replica dynamics}.
\newblock {\em ESAIM: Mathematical Modelling and Numerical Analysis},
  47(5):1287--1314, July 2013.

\bibitem{TadmorOrtizPhillips1996}
E.B. Tadmor, M.~Ortiz, and R.~Phillips.
\newblock Quasicontinuum analysis of defects in solids.
\newblock {\em Philosophical Magazine A}, 73(6):1529--1563, 1996.

\bibitem{Venturini:2014fv}
G.~Venturini, K.~Wang, I.~Romero, M.P. Ariza, and M.~Ortiz.
\newblock {Atomistic long-term simulation of heat and mass transport}.
\newblock {\em Journal Of The Mechanics And Physics Of Solids}, 73(C):242--268,
  December 2014.

\bibitem{longestmd}
Vincent~A. Voelz, Gregory~R. Bowman, Kyle Beauchamp, and Vijay~S. Pande.
\newblock Molecular simulation of ab initio protein folding for a millisecond
  folder {NTL9 (1--39)}.
\newblock {\em Journal of the American Chemical Society}, 132(5):1526--1528,
  2010.
\newblock PMID: 20070076.

\bibitem{Voter:1997p12731}
A.F Voter.
\newblock Hyperdynamics: Accelerated molecular dynamics of infrequent events.
\newblock {\em Phys. Rev. Lett.}, 78(20):3908--3911, Jan 1997.

\bibitem{Voter:1998p13729}
A.F. Voter.
\newblock Parallel replica method for dynamics of infrequent events.
\newblock {\em Phys. Rev. B}, 57(22):13985--13988, Jan 1998.

\bibitem{WagnerLiu2003}
G.~J. Wagner and W.~K. Liu.
\newblock Coupling of atomistic and continuum simulations using a bridging
  scale decomposition.
\newblock {\em Journal of Computational Physics}, 190(1):249--274, 2003.

\bibitem{Xu:2008jk}
L.~Xu and G.~Henkelman.
\newblock {Adaptive kinetic Monte Carlo for first-principles accelerated
  dynamics}.
\newblock {\em J. Chem. Phys.}, 129(11), 2008.

\end{thebibliography}

\end{document}